\documentclass[12pt]{article}
\usepackage{mathrsfs}
\usepackage{amsfonts}

\usepackage{amsmath}
\usepackage{amssymb}
\usepackage{amsthm}
\usepackage{fancyhdr}

\usepackage{todonotes}

\usepackage{geometry}
\theoremstyle{plain}
\newtheorem{theorem}{Theorem}[section]

\newtheorem{lemma}[theorem]{Lemma}

\theoremstyle{definition}

\newtheorem{remark}[theorem]{Remark}

\def\UU{{\mathcal U}}
\def\FF{{\mathcal F}}

\usepackage{latexsym}
\usepackage{amssymb}
\usepackage{fancyvrb}
\usepackage[mathscr]{eucal}
\usepackage{hyperref}

\begin{document}
	
	
	\begin{center}
		{\Large \bf     On the partial $ \Pi  $-property of some subgroups of prime power order of finite groups
			
			\renewcommand{\thefootnote}{\fnsymbol{footnote}}
			
			\footnotetext[1]
			{Corresponding author.}
			
	}\end{center}
	
	\vskip0.6cm
	\begin{center}
		
		Zhengtian Qiu, Jianjun Liu and Guiyun Chen$^{\ast}$
		
		\vskip0.5cm
		
		School of Mathematics and Statistics, Southwest University,
		
		Chongqing 400715, P. R. China
		
		E-mail addresses:  qztqzt506@163.com \, \ liujj198123@163.com    \, \  gychen1963@163.com

	\end{center}
	
	\vskip0.5cm

	\begin{abstract}
		Let $ H $ be a subgroup of a finite group $ G $.  We say that $ H $ satisfies the partial $ \Pi  $-property in $ G $ if  there exists a chief series $ \varGamma_{G}: 1 =G_{0} < G_{1} < \cdot\cdot\cdot < G_{n}= G $ of $ G $ such that for every $ G $-chief factor $ G_{i}/G_{i-1} (1\leq i\leq n) $ of $ \varGamma_{G} $, $ | G / G_{i-1} : N _{G/G_{i-1}} (HG_{i-1}/G_{i-1}\cap G_{i}/G_{i-1})| $ is a $ \pi (HG_{i-1}/G_{i-1}\cap G_{i}/G_{i-1}) $-number.  In this paper, we study the influence of some  subgroups of prime power order satisfying  the partial $ \Pi  $-property  on the structure of a finite group.
	\end{abstract}
	
	{\hspace{0.88cm} \small \textbf{Keywords:} Finite group, $ p $-soluble group, the partial  $ \Pi $-property, $ p $-length.}
	
	\vskip0.1in
	
	{\hspace{0.88cm} \small \textbf{Mathematics Subject Classification (2020):} 20D10,   20D20.}

	\section{Introduction}

	\hspace{0.5cm} All groups considered in this paper are finite.
	We use conventional notions as in \cite{Huppert-1967}. $ G $ always denotes a finite group, $ p $ denotes a fixed prime, $ \pi $ denotes some set of primes  and $ \pi(G) $ denotes the set of all primes dividing $ |G| $. An integer $ n $ is called a $ \pi $-number if all prime divisors of $ n $ belong to $ \pi $.
	
	Recall that a class $ \FF $  of groups is called a formation if $ \FF $ is closed under taking homomorphic images and subdirect products.  A formation $ \FF $ is said to be saturated if $ G/\Phi(G) \in \FF $  implies that $ G \in \FF $.  Throughout  this paper, we  use $ \UU $ (resp. $ \UU_{p} $) to denote the class of supersoluble (resp. $ p $-supersoluble) groups.
	
	Let $ \FF $ be a formation. The $ \FF $-residual of $ G $, denoted by $ G^{\FF} $, is the smallest normal subgroup of $ G $
	with quotient in $ \FF $. A chief factor $ L/K $ of $ G $ is said to be $ \FF $-central  (resp. $ \FF $-eccentric) in $ G $ if $ L/K \rtimes G/C_{G}(L/K) \in \FF $ (resp. $ L/K \rtimes G/C_{G}(L/K) \not \in \FF $). In particular, a chief factor $ L/K $ of  $ G $ is $ \UU $-central in $ G $  if and only if $ L/K $ is cyclic.  A normal subgroup $ N $ of $ G $ is called $ \FF $-hypercentral in $ G $ if either $ N = 1 $ or every $ G $-chief factor below $ N $ is
	$ \FF $-central in $ G $. Let $ Z_{\FF}(G) $ denotes the $ \FF $-hypercenter of $ G $, that is, the product of all $ \FF $-hypercentral normal subgroups of $ G $.

	In \cite{Chen-2013}, Chen and Guo introduced the concept of the partial  $ \Pi $-property of subgroups of finite groups,  which generalizes a large number of known embedding
	properties (see \cite[Section 7]{Chen-2013}).
	A subgroup $H$ of a group $G$ is said to satisfy the partial $ \Pi$-property in $G$  if  there exists a chief series $ \varGamma_{G}: 1 =G_{0} < G_{1} < \cdot\cdot\cdot < G_{n}= G $ of $G$ such that for every $ G $-chief factor $ G_{i}/G_{i-1} $ $ (1\leq i\leq n) $ of $ \varGamma_{G} $, $ | G / G_{i-1} : N _{G/G_{i-1}} (HG_{i-1}/G_{i-1}\cap G_{i}/G_{i-1})| $ is a $ \pi (HG_{i-1}/G_{i-1}\cap G_{i}/G_{i-1}) $-number. They proved the following results by assuming  some maximal or minimal subgroups of a Sylow subgroup  satisfiy the partial $ \Pi $-property.
	
	\begin{theorem}[{\cite[Proposition 1.4]{Chen-2013}}]\label{maximal}
		Let $ E $ be a normal subgroup of $ G $ and let $ P $ be a Sylow $ p $-subgroup of $ E $. If every maximal subgroup of $ P $ satisfies the partial $ \Pi $-property in $ G $, then either $ E\leq Z_{\UU_{p}}(G) $ or $ |E|_{p}=p $.
	\end{theorem}
	
	\begin{theorem}[{\cite[Proposition 1.6]{Chen-2013}}]\label{minimal}
		Let $ E $ be a normal subgroup of $ G $ and let $ P $ be a Sylow $ p $-subgroup of $ E $. If  every cyclic subgroup of $ P $ of prime order or order $ 4 $ (when $ P $ is
		not quaternion-free)  satisfies the partial $ \Pi $-property in $ G $, then  $ E\leq Z_{\UU_{p}}(G) $.
	\end{theorem}

	It is well known that  the $ p $-length of a $ p $-supersoluble
	group is at most $ 1 $ (see \cite[Kapitel IV, Hauptsatz 6.6]{Huppert-1967}). In this paper, we investigate the structure of  a finite group in which some subgroups of prime power order satisfy  the partial $ \Pi $-property. Our  results  are as follows.
	
	\begin{theorem}\label{first}
		Let $ P $ be a Sylow $ p $-subgroup of $ G $, and let $ d $ be a power of $ p $ such that $ 1 < d < |P| $. Assume that every  subgroup of  $ P $ of order $ d $ satisfies the partial $ \Pi $-property in $ G $, and  every cyclic subgroup of $ P $ of order $ 4 $ (when $ d=2 $ and $ P $ is not quaternion-free) satisfies the partial $ \Pi $-property in $ G $.  Then $ G $ is $ p $-soluble with  $ p $-length at most $ 1 $.
	\end{theorem}
	
	In order to prove Theorem \ref{first}, we need the following   results.

	\begin{theorem}\label{second}
		Let $ G $ be a $ p $-soluble group and $ P $ a Sylow $ p $-subgroup of $ G $ and let $ d $ be a power of $ p $ such that $ 1 < d < |P| $. Assume that every  subgroup of  $ P $ of order $ d $ satisfies the partial $ \Pi $-property in $ G $, and  every cyclic subgroup of $ P $ of order $ 4 $ (when $ d=2 $ and $ P $ is not quaternion-free) satisfies the partial $ \Pi $-property in $ G $.  Then $ G $ has  $ p $-length  at most $ 1 $.
	\end{theorem}
	
	\begin{theorem}\label{third}
		Let $ P $ be a Sylow $ p $-subgroup of $ G $, and let $ d $ be a power of $ p $ such that $ 1 < d < |P| $. Assume that every  subgroup of  $ P $ of order $ d $ satisfies the partial $ \Pi $-property in $ G $, and  every cyclic subgroup of $ P $ of order $ 4 $ (when $ d=2 $ and $ P $ is not quaternion-free) satisfies the partial $ \Pi $-property in $ G $.  Then $ G $ is a $ p $-soluble group.
	\end{theorem}

	\begin{remark}
		Actually, we can not obtain the $ p $-supersolubility of $ G $ in Theorem  \ref{first}. For an example, we refer the reader to  \cite[Example 1.5]{Qiu-Liu-Chen}.
			\end{remark}

	\section{Preliminaries}

	\begin{lemma}[{\cite[Lemma 2.1(3)]{Chen-2013}}]\label{over}
		Let $ H \leq G $ and $ N \unlhd G $. If either $ N \leq H $ or $ (|H|, |N|)=1 $ and $ H $ satisfies the partial $ \Pi $-property in $ G $, then $ HN/N $ satisfies the partial $ \Pi $-property in $ G/N $.
	\end{lemma}
	
	\begin{lemma}[{\cite[Lemma 2.2]{Qiu-Liu-Chen}}]\label{subgroup}
		Let  $ H\leq N \leq G $.  If $ H $ is a $ p $-subgroup of $ G $ and $ H $ satisfies the partial $ \Pi $-property in $ G $, then $ H $ satisfies the partial $ \Pi $-property in $ N $.
		\end{lemma}

	\begin{lemma}[{\cite[Lemma 2.3]{Qiu-Liu-Chen}}]\label{pass}
		Let $ H $ be a $ p $-subgroup of $ G $  and $ N $ be  a normal subgroup of $ G $  containing $ H $. If  $ H $ satisfies the partial $ \Pi $-property in $ G $, then $ G $ has  a chief series $$ \varOmega_{G}: 1 =G^{*}_{0} < G^{*}_{1} < \cdot\cdot\cdot <G^{*}_{r}=N < \cdot\cdot\cdot < G^{*}_{n}= G $$  passing through $ N $ such that $ |G:N_{G}(HG^{*}_{i-1}\cap G^{*}_{i})| $ is a $ p $-number  for every $ G $-chief factor $ G^{*}_{i}/G^{*}_{i-1} $ $ (1\leq i\leq n) $ of $ \varOmega_{G} $.
	\end{lemma}

	\begin{lemma}[{\cite[Theorem 1]{Ballester-Bolinches-1996}}]\label{Important}
		Let $ \FF $ be a saturated formation and let $ G $ be a group such
		that $ G $ does not belong to $ \FF $ but all its proper subgroups belong to $ \FF $. Then
		
		\vskip0.08in
		
		\noindent{\rm (1)} $ Z_{\FF}(G) $ is contained in $ \Phi(G) $ and $ F'(G)/\Phi(G) $ is the unique minimal normal subgroup of $ G/\Phi(G) $, where $ F'(G) =
		{\rm Soc}(G$ {\rm mod} $\Phi(G)) $, and $ F'(G) = G^{\FF}\Phi(G) $. Moreover, $ G^{\FF}\Phi(G)/\Phi(G) $ is an $ \FF $-eccentric
 chief factor of $ G $.
				
		\noindent{\rm (2)} If the derived subgroup $ (G^{\FF})' $ of $ G^{\FF} $ is a proper subgroup of
		$ G^{\FF} $, then $ G^{\FF} $ is a soluble group.
		
		\noindent{\rm (3)} If $ G^{\FF} $ is soluble, then $ F'(G) = F(G) $, the Fitting subgroup of  $ G $, and $ Z_{\FF}(G) =
		\Phi(G) $.
			\end{lemma}

	\begin{lemma}[{\cite[Proposition 1]{Ballester-Bolinches-1996}}]\label{important}
		Let $ \FF  $ be a saturated formation and let $ G $ be a group which does not belong to $ \FF $. Suppose that there exists a maximal subgroup $ M $ of $ G $ such that $ M\in \FF  $  and $ G =
		MF(G) $. Then $ G^{\FF}/(G^{\FF})' $ is a chief factor of $ G $, $ G^{\FF} $ is a $ p $-group for some prime $ p $,
		$ G^{\FF} $ has exponent $ p $ if $ p $ is odd and exponent at most $ 4 $ if $ p = 2 $. Moreover, either $ G^{\FF} $
		is elementary abelian or $ (G^{\FF})'= Z(G^{\FF}) = \Phi(G^{\FF}) $ is elementary abelian.
		\end{lemma}

	If $ P $ is either an odd order $ p $-group or a quaternion-free $ 2 $-group, then we use $ \Omega(P) $ to denote the subgroup $ \Omega_{1} (P) $.  Otherwise, $ \Omega (P) = \Omega_{2} (P) $.

	\begin{lemma}\label{hypercenter}
		Let $ P $ be a normal $ p $-subgroup of $ G $ and $ D $ a Thompson critical subgroup of $ P $ \cite[page 186]{Gorenstein-1980}. If $ P/\Phi(P) \leq Z_{\UU}(G/\Phi(P)) $ or  $ \Omega(D) \leq Z_{\UU}(G) $, then $ P \leq  Z_{\UU}(G) $.
	\end{lemma}
	
	\begin{proof}
		By \cite[Lemma 2.8]{Chen-xiaoyu-2016}, the conclusion follows.
	\end{proof}

	\begin{lemma}[{\cite[Lemma 2.10]{Chen-xiaoyu-2016}}]\label{critical}
		Let $ D $ be a Thompson critical subgroup of a non-trivial $ p $-group $ P $.
		
		\vskip0.08in
		
		\noindent{\rm (1)} If $ p > 2 $, then the exponent of $ \Omega_{1}(D) $ is $ p $.
		
		\noindent{\rm (2)} If $ P $ is an abelian $ 2 $-group, then the exponent of $ \Omega_{1}(D) $ is $ 2 $.
		
		\noindent{\rm (3)} If $ p = 2 $, then the exponent of $ \Omega_{2}(D) $ is at most $ 4 $.
	\end{lemma}
	
	\begin{lemma}[{\cite[Lemma 3.1]{Ward}}]\label{charcteristic}
		Let $ P $ be a non-abelian quaternion-free $ 2 $-group. Then $ P $ has a characteristic subgroup of index $ 2 $.
	\end{lemma}

	\begin{lemma}\label{order-d}
		Let $ P $ be a Sylow $ p $-subgroup of $ G $ and let $ d $ be a power of $ p $ such that $ 1 < d < |P| $. Assume that every  subgroup of  $ P $ of order $ d $ satisfies the partial $ \Pi $-property in $ G $. Then:
		
		\vskip0.08in
		
		\noindent{\rm (1)}  Every minimal normal subgroup of $ G $ is either a $ p' $-group or a $ p $-group of order at most $ d $.
		
		\noindent{\rm (2)} If $ G $  has  a minimal normal subgroup  of order $ d $, then  every minimal normal $ p $-subgroup of $ G $ has order $ d $.
	\end{lemma}
	\begin{proof}
		(1) Let $ R $ be a minimal normal subgroup of $ G $ such that $ p\big | |R| $. Let $ L $ be a normal subgroup of $ P $ of order $ p $ contained in  $ P\cap R $  and let $ H $ be a normal subgroup of $ P $ of order $ d $ containing $ L $. Then $ H $ satisfies the partial $ \Pi $-property in $ G $. There exists a chief series $ \varGamma_{G}: 1 =G_{0} < G_{1} < \cdot\cdot\cdot < G_{n}= G $ of $ G $ such that for every $ G $-chief factor $ G_{i}/G_{i-1} $ $ (1\leq i\leq n) $ of $ \varGamma_{G} $, $ | G  : N _{G} (HG_{i-1}\cap G_{i})| $ is a $ p $-number. Notice  that there exists an integer $ j $ $ (1 \leq j \leq n) $ such that $ G_{j}=  G_{j-1} \times R $. Then $ |G  : N _{G} (HG_{j-1}\cap G_{j})| $ is a $ p $-number, and thus $ HG_{j-1}\cap G_{j}\unlhd G $. If $ HG_{j-1}\cap G_{j}=G_{j-1} $, then $  H\cap G_{j}\leq G_{j-1} $. Therefore $ L\leq H\cap R\leq G_{j-1} \cap R=1 $, a contradiction.  Hence $ HG_{j-1}\cap G_{j}=G_{j} $. This implies that $ G_{j-1}R= G_{j}\leq HG_{j-1} $. Consequently,  ${|R|}$ divides $ |H:H\cap G_{j-1}| $. Thus $ R $ is a $ p $-group of order at most $ d $.
		
		(2) Let $ K $ be a minimal normal subgroup of  $ G $ of order $ d $.  In view of (1), we know that every minimal normal $ p $-subgroup of $ G $ has order at most $ d $.  Suppose that there
		exists a minimal normal $ p $-subgroup $ T $ of $ G $ such that $ |T| < d $. Let $ H $ be a normal subgroup of $ P $ of order $ d/|T| $ contained in $ K $.  Then $ HT $ is a subgroup of $ KT $ of order $ d $. By hypothesis, $ HT $ satisfies the partial $ \Pi $-property in $ G $. Applying Lemma \ref{pass}, $ G $ has  a chief series $$ \varOmega_{G}: 1 =G^{*}_{0} < G^{*}_{1} < G^{*}_{2}=KT < \cdot\cdot\cdot < G^{*}_{n}= G $$  passing through $ KT $ such that $ |G:N_{G}(HTG^{*}_{i-1}\cap G^{*}_{i})| $ is a $ p $-number  for every $ G $-chief factor $ G^{*}_{i}/G^{*}_{i-1} $ $ (1\leq i\leq n) $ of $ \varOmega_{G} $. Since $ KT/T $ and $ G^{*}_{1}T/T $ are minimal normal subgroups of $ G/T $ and $ G^{*}_{1}T\leq KT $,
		we have that either $ |G^{*}_{1}| = |K| = d $ or $ T=G^{*}_{1} $. Consider the $ G $-chief factor $ KT/G^{*}_{1} $. If $ T=G^{*}_{1} $, then $ |G:N_{G}(HT\cap KT)| $ is a $ p $-number, and  so $ HT=HT\cap KT\unlhd G $. This  implies that $ HT=T $ or $ HT=KT $, a contradiction. Hence $ |G^{*}_{1}|=|K|=d $. Consider the $ G $-chief factor $ G^{*}_{1}/1 $. Then $|G:N_{G}(HT\cap G^{*}_{1})|$  is a $ p $-number, and hence $ HT\cap G^{*}_{1}\unlhd G $. Since $ T $ and $ G^{*}_{1} $ are minimal normal subgroups of $ G $ of
		different orders, we have $ G^{*}_{1}\nleq HT $. As a consequence, $ HT\cap G^{*}_{1}=1 $ and $ |(HT)G^{*}_{1}|=d^{2} $. Note that $ (HT)G^{*}_{1} $ is a subgroup of $ KT $. It follows that $ |(HT)G^{*}_{1}| $ divides $ |K||T|<d^{2} $. This is a contradiction. Hence  every minimal normal $ p $-subgroup of $ G $ has order $ d $. The proof is thus complete.
	\end{proof}

	\begin{lemma}\label{unique-maximal}
		Let $ N \unlhd G $. If $ M \unlhd G $ is minimal such that $ M/(M\cap N)\nleq Z_{\UU_{p}}(G/(M\cap N)) $, then $ M $ has a unique maximal $ G $-invariant subgroup, $ L $ say; also, $ L \geq M \cap N $,
		$ L/(M \cap N)\leq Z_{\UU_{p}}(G/(M\cap N)) $, $ M/L\nleq Z_{\UU_{p}}(G/L) $.
	\end{lemma}
	\begin{proof}
		Assume that $ M $ has two different maximal $ G $-invariant subgroups,  $ L_{1} $, $ L_{2} $ say. By the minimality of $ M $, $ L_{i}/(L_{i}\cap N)\leq Z_{\UU_{p}}(G/(L_{i}\cap N)) $  for each $ i \in \{ 1, 2 \} $. This implies that $ L_{1}N/N $ and $ L_{2}N/N $  are contained in $ Z_{\UU_{p}}(G/N) $. Consequently $ M/(M \cap N)\cong MN/N = (L_{1}N/N)(L_{2}N/N) \leq Z_{\UU_{p}}(G/N) $, a contradiction. Hence $ M $ has  a unique maximal $ G $-invariant subgroup,  $ L $ say. Clearly,  $ M > M \cap N $. Now the uniqueness of $ L $ implies $ L\geq M \cap N $. Since $ L/(L \cap N)\leq Z_{\UU_{p}}(G/(L\cap N)) $  and $ L \cap N = M \cap N $, we deduce that $ L/(M \cap N)\leq Z_{\UU_{p}}(G/(M\cap N)) $. Hence   $ M/L\nleq Z_{\UU_{p}}(G/L) $.
	\end{proof}
	
	\begin{lemma}\label{Up}
		Let $ P $  be a normal $ p $-subgroup  of $ G $.   If all cyclic subgroups of $ P $ of order $ p $ or $ 4 $ (when   $ P $ is  not quaternion-free) are contained in $ Z_{\UU}(G) $, then $ P\leq Z_{\UU}(G) $.
	\end{lemma}
	
	\begin{proof}
		Assume that this lemma is not true and let $ ( G, P ) $ be a counterexample for which $ | G || P | $ is minimal.
		
		Let $ N $ be a maximal $ G $-invariant subgroup of  $ P $. Clearly,  the hypotheses are inherited by $ (G, N) $. The minimal choice of $ (G, P) $ implies that  $ N\leq Z_{\UU}(G) $. Obviously, $ N $ is the  unique maximal $ G $-invariant subgroup of $ P $.   We claim that the exponent of $ P $ is $ p $ or $ 4 $ (when $ P $  is not quaternion-free). If $ P $ is either an odd order $ p $-group or a quaternion-free $ 2 $-group, then we use $ \Omega(P) $ to denote the subgroup $ \Omega_{1} (P) $.
		Otherwise, $ \Omega (P) = \Omega_{2} (P) $.  Let $ D $ be a Thompson critical subgroup of $ P $.  If $ \Omega(D)< P $, then $ \Omega(D) \leq Z_{\UU}(G) $.
		By Lemma \ref{hypercenter}, we have that $ P\leq Z_{\UU}(G) $, a  contradiction. Thus $ P = D = \Omega(D) $.  If $ P $ is a non-abelian quaternion-free $ 2 $-group, then $ P $ has a characteristic subgroup $ R $ of index $ 2 $ by Lemma \ref{charcteristic}. The uniqueness
		of $ N $ yields that $ N=R $. Hence $ P\leq Z_{\UU}(G) $, a contradiction.  This means that $ P $ is a non-abelian $ 2 $-group if and only if $ P $ is not quaternion-free.  In view of  Lemma \ref{critical}, the exponent of $ P $ is $ p $ or $ 4 $ (when $ P $  is not quaternion-free), as claimed. Therefore, every non-identity element of $ P $ has order $ p $ or $ 4 $ (when $ P $  is not quaternion-free). It follows that $ P\leq Z_{\UU}(G) $, a contradiction. The proof of the lemma is complete.
	\end{proof}
	
	\begin{lemma}\label{hyper}
		Let $ N $ be a normal subgroup of $ G $ and $ P\in {\rm Syl}_{p}(N) $. Then $ N\leq Z_{\UU_{p}}(G) $ if and  only if
		all cyclic subgroups of $ P $ of order $ p $ or $ 4 $ (when   $ P $ is  not quaternion-free) are contained in $ Z_{\UU_{p}}(G) $.
	\end{lemma}
	\begin{proof}
		We only need to prove the sufficiency. Assume that $ N\nleq Z_{\UU_{p}}(G) $  and let $ ( G , N ) $ be a minimal counterexample for which $ |G||N| $ is minimal.
		
		\vskip0.1in
		
		\noindent\textbf{Step 1.} $ O_{p'}(N)=1 $.
		
		\vskip0.1in
		
		Clearly,  the hypotheses are
		inherited by $ (G/O_{p'}(N), N/O_{p'}(N)) $.  The minimal choice of $ (G, N) $ yields that $ O_{p'}(N) = 1 $.
		
		\vskip0.1in

		\noindent\textbf{Step 2.}  $ N=G $.
		
		\vskip0.1in
		
		All cyclic subgroups of $ P $ of order $ p $ or $ 4 $ (when  $ P $ is  not quaternion-free) are contained in $ Z_{\UU_{p}}(G)\cap N\leq Z_{\UU_{p}}(N) $. Suppose that $ N < G $. Then $ (N, N) $ satisfies the hypothesis. By  the choice of $ (G, N) $, we see that $ N $ is $ p $-supersoluble. Since $ O_{p'}(N)=1 $, it follows from \cite[Lemma 2.1.6]{Ballester-2010} that $ P\unlhd N $. Thus $ P\unlhd G $. By Lemma  \ref{Up}, $ P\leq Z_{\UU_{p}}(G) $, and hence $ N\leq Z_{\UU_{p}}(G) $, a contradiction. Therefore $ N=G $.
		
		\vskip0.1in
		\noindent\textbf{Step 3.} Let $ L $ be a maximal normal subgroup of $ G $. Then $ L\leq Z_{\UU_{p}}(G) $. Moreover, $ L=Z(G)=Z_{\UU_{p}}(G) $ is the unique maximal  normal subgroup of $ G $.
		
		\vskip0.1in

		By Step 2, $ (G, L)  $ satisfies the hypotheses. The minimal choice of $ (G, N) $ yields that $ L\leq Z_{\UU_{p}}(G) $, and so $ L=Z_{\UU_{p}}(G) $. If $ T $ is a maximal normal subgroup of $ G $ which is different from $ L $, then $ T\leq Z_{\UU_{p}}(G) $ with a similar argument. Hence $ G=LT\leq Z_{\UU_{p}}(G) $, a contradiction. This means  that $ L $ is  the unique maximal normal subgroup of $ G $.  If $ G^{\UU_{p}} < G $, then $ G^{\UU_{p}} \leq Z_{\UU_{p}}(G) $.  It follows that $ G=Z_{\UU_{p}}(G) $,  a contradiction. Therefore $ G^{\UU_{p}}=G $. By \cite[Chap. IV, Theorem 6.10]{Doerk-Hawkes},  we get  that $ Z_{\UU_{p}}(G)\leq Z(G) $,  and thus  $ Z_{\UU_{p}}(G)=Z(G) $.

		\vskip0.1in
				
		\noindent\textbf{Step 4.}  The final contradiction.
		
		\vskip0.1in
		
		By Step 3, all cyclic subgroups of $ P $ of order $ p $ or $ 4 $ (when   $ P $ is  not quaternion-free) are contained in $ Z(G) $. By \cite[Kapitel IV, Satz 5.5]{Huppert-1967}, we may assume that  $ p =2 $ and $ P $ is quaternion-free. By \cite[Corollary 2]{Asaad-2004}, we conclude that $ G $ is $ 2 $-nilpotent. This final contradiction completes the proof.
	\end{proof}
	
	\begin{lemma}\label{p-supersoluble}
		Let $ P $ be a Sylow $ p $-subgroup of $ G $, and let $ d $ be a power of $ p $ such that $ p^{2} < d < |P| $. Assume that  $ O_{p}(G) $ is the  unique maximal normal subgroup of $ G $,  every minimal normal subgroup of $ G $ is a $ p $-group of order $ d $, and every  subgroup of  $ P $ of order $ d $ satisfies the partial $ \Pi $-property in $ G $, then  $ G/{\rm Soc}(G) $ is $ p $-supersoluble.
	\end{lemma}
	\begin{proof}
		Assume the result is false. Write  $ N={\rm Soc}(G) $. Let $ M \unlhd  G $ be  minimal such that $ M/(M \cap N)\nleq Z_{\UU_{p}}(G/(M\cap N)) $.   Clearly,  $ M\cap N\not =1 $ is the product of all minimal $ G $-invariant subgroups of $ M $. By Lemma \ref{unique-maximal}, $ M $ admits a unique maximal $ G $-invariant subgroup,  $ U $ say, also $ U/(M\cap N)\leq Z_{\UU_{p}}(G/(M\cap N)) $, while $ M/U\nleq Z_{\UU_{p}}(G/U) $.
		
		Assume that $ M=G $. Then $ U=O_{p}(G) $ and  $ Z_{\UU_{p}}(G/N)=O_{p}(G)/N $. Applying Lemma \ref{hyper} to  $ G/N $,   $ G/N $ possesses a cyclic subgroup  $ \langle x \rangle N/N  $ of order $ p $ or $ 4 $ such that $ x\in G-U $.  Let $ x_{0} \in G - U   $ be a $ p $-element such that $ o(x_{0}) $ is as small as possible.   Since $ o(x_{0}) \leq o(x) $ and $ N $ is elementary abelian, we have that
		$ o(x_{0}) \leq p^{3} $. Note that $ x_{0}^{p}\in U $ by the minimality of $ o(x_{0}) $. Clearly, $ o(x_{0}) \leq d  $. Let $ H $ be a subgroup of $ P $ of order $ d $ such that $ \langle x_{0} \rangle\leq  H\leq \langle x_{0} \rangle N $. 	Since $ x_{0} \in
		G - U $, the uniqueness of $ U $ implies $ H^{G} = G $ and  $ (HU)^{G}=G $. By hypothesis,  $ H $ satisfies the partial $ \Pi $-property in $ G $. For the $ G $-chief factor $ G/U $, we have $ |G:N_{G}(HU)| $ is a $ p $-number. Therefore $ G=N_{G}(HU)P $.  This forces that $ G= (HU)^{G}=H^{P}U $. Hence $ G/U $ is a $ p $-group. This is a contradiction.
		
		Assume that $ M<G $. Then $ M\leq O_{p}(G) $. Write $ \overline{G}=G/(M\cap N) $. We claim that the exponent of $ \overline{M} $ is $ p $ or $ 4 $ (when $ \overline{M} $  is not quaternion-free). If $ \overline{M} $ is either an odd order $ p $-group or a quaternion-free $ 2 $-group, then we use $ \Omega(\overline{M}) $ to denote the subgroup $ \Omega_{1} (\overline{M}) $.
		Otherwise, $ \Omega (\overline{M}) = \Omega_{2} (\overline{M}) $.  Let $ \overline{D} $ be a Thompson critical subgroup of $ \overline{M} $.  If $ \Omega(\overline{M})< \overline{M} $, then $ \Omega(\overline{M}) \leq Z_{\UU}(\overline{G}) $.
		By Lemma \ref{hypercenter}, we have that $ \overline{M}\leq Z_{\UU}(\overline{G}) $, a  contradiction. Thus $ \overline{M} = \overline{D} = \Omega(\overline{D}) $.  If $ \overline{M} $ is a non-abelian quaternion-free $ 2 $-group, then $ \overline{M} $ has a characteristic subgroup $ \overline{M_{0}} $ of index $ 2 $ by Lemma \ref{charcteristic}. The uniqueness of $ U $ implies $ M_{0}=U $. Thus  $ \overline{M} \leq Z_{\UU}(\overline{G}) $, which is impossible. This means that $ \overline{M} $ is a non-abelian $ 2 $-group if and only if $ \overline{M} $ is not quaternion-free.  In view of  Lemma \ref{critical}, the exponent of $ \overline{M} $ is $ p $ or $ 4 $ (when $ \overline{M} $  is not quaternion-free), as claimed.
		
		There exists  a normal subgroup $ \overline{Q} $ of $ \overline{P} $  such that $ \overline{U}< \overline{Q}< \overline{M} $ and $ |\overline{Q}:\overline{U}|=p $.  Pick a non-identity  element $ \overline{y}\in \overline{Q}-\overline{U} $. Then $ \langle \overline{y} \rangle \overline{U}=\overline{Q} $ and  $ |\langle \overline{y} \rangle|\leq p^{2}$. Since $ M\cap N $ is  elementary abelian, we have  $ o(y)\leq p^{3} $.  Let $ H $ be a  subroup of $ Q $ of  order  $ d $ such that $ \langle y \rangle \leq H $. Then $ Q=HU $. By hypothesis, $ H $ satisfies the partial $ \Pi  $-property in $ G $. Applying Lemma \ref{pass}, $ G $ has  a chief series
		$$ \varOmega_{G}: 1 =G^{*}_{0} < G^{*}_{1} < \cdot\cdot\cdot <G^{*}_{r-1} <G^{*}_{r}=M < \cdot\cdot\cdot < G^{*}_{n}= G $$
		passing through $ M $ such that $ |G:N_{G}(HG^{*}_{i-1}\cap G^{*}_{i})| $ is a $ p $-number  for every $ G $-chief factor $ G^{*}_{i}/G^{*}_{i-1}$ $ (1\leq i\leq n) $ of $ \varOmega_{G} $. The uniqueness of $ U $ yields $ G_{r-1}^{*}=U $. Therefore $ |G:N_{G}(HU\cap M|=|G:N_{G}(Q) $ is a $ p $-number. It follows that $ Q\unlhd G $, which contradicts the fact that $ U $ is the  unique maximal $ G $-invariant subgroup of  $ M $. Hence the lemma is proved.
	\end{proof}

	\section{Proofs}

	\begin{proof}[Proof of Theorem \ref{second}]
		Assume  that the theorem is not true and $ G $ is a counterexample of  minimal  order. Let $ \FF $ denote the class of all $ p $-soluble groups whose $ p $-length is at most $ 1 $. Then  $ \FF $ is a saturated formation by \cite[Chap. IV, Examples 3.4]{Doerk-Hawkes}. We divide the proof into the following steps.
		
		\vskip0.1in
		
		\noindent\textbf{Step 1.} $ O_{p'}(G) = 1 $ and $ O^{p'}(G) = G $. Then $ O_{p}(G)>1 $ and $ O^{p}(G)<G $. In particular, $ F(G) $ is a $ p $-group.
		
		\vskip0.1in
		
		Applying Lemma \ref{over}, $ G/O_{p'}(G) $ satisfies  the hypotheses of the theorem. If $ |G/O_{p'}| < |G| $, it would follow that $  G/O_{p'} \in \FF $, and thus $ G \in \FF $. This is a contradiction. Hence $ O_{p'}(G)=1 $. Since $ G $ is $ p $-soluble, we have $ O_{p}(G)>1 $.
		
		Assume that $ O^{p'}(G) < G $. Since  $ O^{p'}(G) $ satisfies the hypotheses of the
		theorem by Lemma \ref{subgroup}, we deduce  that $ O^{p'}(G)\in \FF $. Thus $ G\in \FF $, a contradiction. Therefore $ O^{p'}(G) = G $. Hence $ O^{p}(G)<G $ since  $ G $ is $ p $-soluble.
		
		\vskip0.1in
		
		\noindent\textbf{Step 2.} $ p^{2}<d<|P|/p^{2} $.
		
		\vskip0.1in
		
		Applying Theorem \ref{maximal} and Theorem \ref{minimal}, $ d\not =p,  |P|/p $. By \cite[Theorem 1.3 and Theorem 1.4]{Qiu-Liu-Chen}, we have  $ d\not =p^{2},  |P|/p^{2} $.
		
		\vskip0.1in
		
		\noindent\textbf{Step 3.} $ \Phi(G)=1 $.
		
		\vskip0.1in
		
		Suppose that $ \Phi(G) > 1 $ and we work to obtain a contradiction.  By Step 1, we see that  $ \Phi(G) $ is a $ p $-group.
		
		\vskip0.1in
		
		\textbf{(3.1)} $ d\leq \Phi(G) $.
		
		\vskip0.1in
		
		Suppose  that $ |\Phi(G)| < d $. Let $ K $ be a minimal normal subgroup of $ G $ contained in $ \Phi(G) $.  Then $ |K| < d $. Assume that $ d/|K|>p $. Then $ p<d/|K|<|P/K| $. If $ H/K $ is a subgroup of $ P/K $ of order $ d/|K| $, then $ H $ is a subgroup of $ P $ of order $ d $. By hypothesis, $ H $ satisfies the partial $ \Pi $-property in $ G $. By Lemma \ref{over}, $ H/K $ satisfies the partial $ \Pi $-property in $ G/K $. The minimal choice of $ G $ yields $ G/K \in \FF$. Therefore, $ G\in \FF $, contradicting the choice of $ G $. Hence $ d/|K|=p $, and thus $  \Phi(G) = K $.
		
		Assume that $ d/|K|=p $.  Then every subgroup of $ P/K $ of order $ p $ satisfies the partial $ \Pi $-property in $ G $. If $ p $ is  odd or $ P/K $  is a  quaternion-free $ 2 $-group, then by Theorem \ref{minimal}, we deduce that  $ G/K\in \FF $, and so $ G\in\FF $, a contradiction. As a consequence, $ p = 2 $ and $ P/K $  is not  quaternion-free. Thus  $ d=2|K| $.

		Assume that there exists a minimal normal subgroup $ T $ of $ G $ such that $ K \not = T $.  By Step 1, $ TK/K $ is a minimal normal $ 2 $-subgroup of $ G/K $. It is clear that every subgroup  of $ P/K $ of order $ 2 $  satisfies the partial $ \Pi $-property in $ G/K $. By Lemma \ref{order-d}, we have that $ |T|=|TK/K|=2 $. By Step 2, $ d/|T|>2 $. Applying Lemma \ref{over}, we  see that every subgroup of $ P/T $ of order $ d/|T| $ satisfies the partial $ \Pi $-property in $ G/T $. The  minimal choice  of $ G $ implies that $ G/T\in \FF $. Since $ \FF $ is a formation, $ G $ possesses at most two minimal normal subgroups, namely $ T $ and $ K = \Phi(G) $.  Note that $ T\leq Z(G) $. There exists a maximal subgroup $ M $ of $ G $ such that $ G=TM $ and $ T\cap M=1 $. Hence $ G=T\times M $. Since $ M\cong G/T\in \FF $, we have $ G \in \FF. $ This contradiction shows that $ \Phi(G)=K = {\rm Soc}(G) $ is the unique minimal normal subgroup of  $ G $. Then $ \Phi(G)=K $ occurs in every chief series of $ G $.
		
		Now pick a cyclic subgroup $ X/K $  of $ P/K $ of order $ 4 $. If $ K\leq \Phi(X)  $, then $ X $ is cyclic, and so $ K $ is cyclic. Thus $ |K| = 2 $ and  $ d = 4 $, which contradicts Step 2. Hence we suppose that $ K\nleq \Phi(X) $. Then there exists a maximal subgroup $ V $ of $ X $ such that $ X = VK $. Then $ V $ has order $ d $ and satisfies the partial $ \Pi $-property in $ G $.  Note that $ K $ is the unique minimal normal subgroup of $ G $. By a routine check, we see that $ X/K = VK/K $ satisfies the partial $ \Pi $-property in $ G/K $.  Now applying Theorem \ref{minimal}, we have that  $ G $ is $ 2 $-supersoluble. Since $ O_{p'}(G)=1 $, it follows from  \cite[Lemma 2.1.6]{Ballester-2010} that $ P\unlhd G $, a contradiction.  Hence  (3.1) holds.

		\vskip0.1in
		
		\textbf{(3.2)} Every proper subgroup of $ G $ belongs to  $ \FF $, and there exists a maximal subgroup $ M $ of $ G $
		such that $ G = MG^{\FF}  $, $ M\in \FF $.  Moreover, $ F(G)=G^{\FF}\Phi(G) $,  $ G^{\FF}/\Phi(G^{\FF}) $ is a chief factor of $ G $ and the
		exponent of $ G^{\FF} $ is $ p $ or at most $ 4 $ if $ p = 2 $.
		
		\vskip0.1in
		
		Let $ U $ be a maximal subgroup of $ G $. Then $ \Phi(G)\leq S $ for any  Sylow $ p $-subgroup $ S $ of $ U $.  By (3.1), we see that $ d\leq |S| $. If $ |S|=d $, then $ S = \Phi(G) $ and $ U\in \FF $. If $ |S| > d $, then $ U $  satisfies the hypotheses of the theorem.   The minimal choice  of $ G $ implies  that $ U\in \FF $.  Therefore every proper subgroup of $ G $ belongs to $ \FF $. By Lemma \ref{Important}, $ G^{\FF}\Phi(G)/\Phi(G) $ is  the  unique minimal normal subgroup  of $ G/\Phi(G) $. Note that $ G $ is $ p $-soluble.  If $ G^{\FF}\Phi(G)/\Phi(G) $ is  a $ p' $-group, then a Hall $ p' $-subgroup of $ G^{\FF}\Phi(G) $ is normal in $ G $ by \cite[Chap. 1, Lemma 1.8.1]{Guo-2000}, which contradicts Step 1. Therefore
		$ G^{\FF}\Phi(G)/\Phi(G) $ is a $ p $-group. By Lemma \ref{Important}, we have $ F(G)=G^{\FF}\Phi(G) $.  Moreover, $ G $ possesses  a maximal subgroup $ M $  such that $ G=MG^{\FF} $.  By Lemma \ref{important},  $ G^{\FF}/\Phi(G^{\FF})=G^{\FF}/(G^{\FF})' $ is a chief factor of $ G $ and the exponent of $ G^{\FF} $ is $ p $ or at most $ 4 $ if $ p = 2 $.
		
		\vskip0.1in
		
		\textbf{(3.3)} Finishing the proof of Step 3.
		
		\vskip0.1in
		
		By Lemma \ref{Important} and \cite[Chap. 1, Lemma 1.14]{Guo2015}, $ G^{\FF}\Phi(G)/\Phi
		(G)\cong G^{\FF}/(G^{\FF}\cap \Phi(G)) $ is an $ \FF $-eccentric chief factor of $ G $. Every chief factor of
		$ G $ below $ \Phi(G) $ is $ \FF $-central in $ G $ since $ Z_{\FF}(G)=\Phi(G) $ by Lemma \ref{Important}.  If $ G $ has a chief series  $$ \varGamma_{G}: 1 =G_{0} < G_{1} < G_{2}< \cdot\cdot\cdot < G_{j}< G^{\FF}<\cdot\cdot\cdot <G_{n}=G $$   passing through $ G^{\FF} $, then by \cite[Theorem 1.2.36]{Ballester-2006}, $ G^{\FF} /G_{j} $ is the unique $ \FF $-eccentric chief factor of $ G $ in $ \varGamma_{G} $. Then $  Z_{\FF}(G) \cap G^{\FF} = \Phi(G) \cap G^{\FF}=G_{j} $. By Lemma \ref{important},  $ G^{\FF}/\Phi(G^{\FF}) $ is a chief factor of $ G $, we have that $ \Phi(G^{\FF}) = \Phi(G) \cap G^{\FF} = Z_{\FF}(G) \cap G^{\FF}  $. Hence $ G_{j}= Z_{\FF}(G) \cap G^{\FF} = \Phi(G) \cap G^{\FF} = \Phi(G^{\FF}) $. By Lemma \ref{important}, either $ G^{\FF} $ is elementary abelian or $ (G^{\FF})' = Z(G^{\FF}) = \Phi(G^{\FF}) $ is elementary abelian.
		
		Suppose  that $ G^{\FF} $ is elementary abelian. Then $ \Phi(G^{\FF})=1 $. Hence $ F(G) = G^{\FF} \times \Phi(G) $ and  $ G^{\FF} $ is a minimal normal subgroup of $ G $. Let $ R $ be a minimal normal subgroup of $ G $ contained in $ \Phi(G) $. Suppose that $ p \leq |G^{\FF}| < d $ and $ p \leq |R| < d $. If $ p $ is odd, then $ G/R $ satisfies the hypotheses of the theorem by Lemma \ref{over}. The  minimal choice of $ G $ yields that $ G/R\in \FF $. Since $ \FF $ is saturated, we  deduce that $ G\in \FF $, a contradiction. Hence  $ p = 2 $. Note that $ G^{\FF}/1 $ is an $ \FF $-eccentric chief factor of $  G $. This yields that $ G^{\FF}>2 $. If
		$ d/|R|>2 $, then by Lemma \ref{over}, the hypotheses are inherited by $ G/R $. With a similar argument as before, we get that  $ G \in \FF $,  which is absurd. Therefore $ d = 2|R| $. Let $ V $ be a normal  subgroup of $ P $ of order $ 2 $ contained in $ G^{\FF} $. Then  $  VR $ has  order $ d $  and so $ VR $ satisfies the partial $ \Pi $-property in $ G $. By Lemma \ref{pass},  $ G $ has  a chief series $$ \varOmega_{G}: 1 =G^{*}_{0} < G^{*}_{1} < G^{*}_{2}=G^{\FF}R < \cdot\cdot\cdot < G^{*}_{n}= G $$  passing through $ G^{\FF}R $ such that $ |G:N_{G}(VRG^{*}_{i-1}\cap G^{*}_{i})| $ is a $ p $-number  for every $ G $-chief factor $ G^{*}_{i}/G^{*}_{i-1} $ $ (1\leq i\leq n) $ of $ \varOmega_{G} $.  Then $ G:N_{G}(VR\cap G^{*}_{1})| $ is a $ p $-number. Thus $ VR\cap G^{*}_{1}\unlhd G $. By similar reasoning, we have that $ VRG^{*}_{1}=VRG^{*}_{1}\cap G^{\FF}R\unlhd G $, and so $ |VRG^{*}_{1}|=|G^{\FF}R| $. If $ VR\cap G^{*}_{1}=1 $, then $ 2|G^{*}_{1}|=|G^{\FF}| $. This is impossible. If $ VR\cap G^{*}_{1}=G^{*}_{1} $, then $ VR\unlhd G $, and thus $ VR=G^{\FF}R $, a contradiction. Therefore, either $ |G^{\FF}|\geq d $ or $ |R|\geq d $.  By (1) and (2) of Lemma \ref{order-d}, we deduce that  $ |G^{\FF}|=|R|=d $. Let $ X $ be a normal subgroup of $ P $ of order $ 2 $ contained in $ G^{\FF} $ and $ Y $ be a normal subgroup of $ P $ of order $ d/2 $ contained in $ R $. Then $ XY $ is a normal subgroup of $ P $ of order $ d $. By hypothesis, $ XY $ satisfies the $ \Pi $-property in $ G $.  By Lemma \ref{pass}, $ G $ has  a chief series $$ \varPsi_{G}: 1 =G^{\sharp}_{0} < G^{\sharp}_{1} <G^{\sharp}_{2}=G^{\FF}R < \cdot\cdot\cdot < G^{\sharp}_{n}= G $$  passing through $ G^{\FF}R $ such that $ |G:N_{G}(XYG^{\sharp}_{i-1}\cap G^{\sharp}_{i})| $ is a $ p $-number  for every $ G $-chief factor $ G^{\sharp}_{i}/G^{\sharp}_{i-1} $ $ (1\leq i\leq n) $ of $ \varPsi_{G} $. In particular, $ |G:N_{G}(XY\cap G^{\sharp}_{1})| $ is a $ p $-number, and so $ XY\cap G^{\sharp}_{1}\unlhd G $. If $ 1<XY\cap G^{\sharp}_{1} $, then $  G^{\sharp}_{1}=XY $ is a minimal normal subgroup of $ G $. This forces that  $ 1<XY\cap R<R $, a contradiction. Hence $ XY\cap G^{\sharp}_{1}=1 $, and so $ XYG^{\sharp}_{1}=G^{\FF}R $. If $ G^{\sharp}_{1}=G^{\FF} $, then $ 1<XY\cap G^{\FF}=XY\cap G^{\sharp}_{1}=1 $, a contradiction. Hence $ G^{\sharp}_{1}\cap G^{\FF}=1 $ and  $ G^{\sharp}_{1}G^{\FF}=G^{\FF}R $. This implies that $ G^{\sharp}_{1}\cong R $ is $ \FF $-central in $ G $.  If $ G^{\sharp}_{1}=R $, then $ 1<XY\cap R=XY\cap G^{\sharp}_{1}=1 $, a contradiction.   Therefore, $ G^{\sharp}_{1}\cap R=1 $ and  $ G^{\sharp}_{1}R=G^{\FF}R $. It follows that $ G^{\sharp}_{1}\cong G^{\FF} $ is $ \FF $-eccentric in $ G $, which contradicts the fact that  $ G^{\sharp}_{1} $ is $ \FF $-central in $ G $.
		
		Suppose that $ (G^{\FF})' = Z(G^{\FF}) = \Phi(G^{\FF}) = \Phi(G) \cap G^{\FF} $ is elementary abelian. Since $ G^{\FF}/\Phi(G^{\FF}) $ is $ \FF $-eccentric in $ G $, we see that $ |G^{\FF}/\Phi(G^{\FF})|>p $. Let $ N $ be a minimal normal subgroup of $ G $ contained in $ \Phi(G^{\FF}) $.  Then $ N\leq Z_{\FF}(G)=\Phi(G) $. With a similar argument in the previous paragraph, we can obtain that either $ |N| = d $ or $ p = 2 $ and $ 2|N| = d $. In particular,  $ |G^{\FF}|>d $.

		Assume that $ |N| = d $.  Let $ A $ be a normal subgroup of $ P $ contained in $ G^{\FF} $ such that $ |A/\Phi(G^{\FF})| = p $. Then $ A<G^{\FF} $. Let $ B $ be a normal subgroup of $ P $ of order $ d/p $ contained in $ N $.
		Pick a non-identity element $ x \in A - \Phi(G^{\FF}) $, then $ o(x) = p $ or $ o(x) = 4 $ by (3.2). If $ o(x) = p $, then $ \langle x \rangle B $ has order $ d $. By hypothesis, $ \langle x \rangle B $ satisfies partial $ \Pi $-property in $ G $. Notice that the chief factor $ G^{\FF}/\Phi(G^{\FF}) $ occurs in any chief series of $ G $ passing through $ G^{\FF} $. By Lemma \ref{pass}, we have $ |G:N_{G}(\langle x \rangle B\Phi(G^{\FF})\cap G^{\FF})| $ is a $ p $-number. Since $ \langle x \rangle B\unlhd P $, we have $ A=\langle x \rangle B\Phi(G^{\FF})\cap G^{\FF}\unlhd G $, a contradiction. Hence we may suppose that $ o(x)=4 $. Notice that $ 4 < d $ by Step 2. Let $ C $ be a normal subgroup of $ P $ of order $ d/4 $ contained in $ B $. If $ x^{2} \in  B  $, then $ \langle x \rangle B $ has order $ d $ and if $ \langle x\rangle \cap  C = 1  $, then $ \langle x\rangle C $ has order $ d $. In both cases, we can argue as above to derive a contradiction. Assume  that $ p=2 $ and $ 2|N| = d $. Then we can handle it in a similar way.  Hence $ \Phi(G)=1 $.

		\vskip0.1in
		
		\noindent\textbf{Step 4.} $ F(G) = O_{p}(G) $ is a minimal normal subgroup of $ G $.
		
		\vskip0.1in
		
		By Step 3 and \cite[Lemma 2.13]{Skiba-2007}, $ O_{p}(G) = T_{1} \times \cdot\cdot\cdot \times T_{s} $, where $ T_{i}$ $ (i = 1,2,...,s) $ is a minimal normal subgroup of $ G $. Assume that $ s\geq 2 $, and we work to obtain a contradiction.  By Lemma \ref{order-d}(1), we see  that $ |T_{i}| \leq d $ for   $i = 1,2,...,s$. Let $ L_{i} $ be a maximal subgroup of $ G $ such that $ G = T_{i}\rtimes L_{i} $ for  $i = 1,2,...,s$. Assume that $ G $ possesses a minimal normal subgroup $ N_{1} $ say, such that $ |N_{1}| = d $. By Lemma \ref{order-d}(2), $ |T_{1}|=|T_{2}|=\cdot\cdot\cdot=|T_{s}|=d $. Hence $ d $ divides  $ L_{i} $ for  $i = 1,2,...,s$. By Lemma \ref{subgroup}, $ L_{i} $ satisfies  the hypotheses  of the theorem for  $i = 1,2,...,s$. The minimal choice of $ G $ implies that $ L_{i}\in \FF $ and  $ G/T_{i} \in \FF $ for  $i = 1,2,...,s$. Since $ \FF $ is a formation, we conclude that $ G \in \FF $, a contradiction.
		
		As  a consequence, we may assume that $ p \leq |T_{i}| < d $ for every $ i $. If $ p $ is odd, then $ G/T_{i} $ satisfies the hypotheses of the theorem for every $ i $. Hence $ G/T_{i}  \in \FF$  for every $ i $. It follows that $ G\in \FF $, a contradiction. Therefore,  $ p = 2 $ and $ d > 4 $ by Step 2. If
		$ d/|T_{i}|> 2 $ for every $ i $, it follows that the hypotheses  are inherited by $ G/T_{i} $. Arguing as before, $ G\in \FF $, which is absurd. Hence there exists a minimal normal subgroup $ T_{1} $ say, such that  $ d = 2|T_{1}| $. Then $ |T_{2}|\leq |T_{1}| $.  We argue that $ |T_{2}|>2 $. It is no loss to assume that $ |T_{2}|<|T_{1}| $. Then $ d/|T_{2}|>2 $.  Let $ H/T_{2} $ be a subgroup of $ P $ of order $ d $. Then $ H $ has order $ d $. By  hypothesis, $ H $ satisfies the  partial $ \Pi $-property in $ G $. By Lemma \ref{over}, $ H/T_{2} $ satisfies the  partial $ \Pi $-property in $ G/T_{2} $. The minimal choice of $ G $ yields that $ G/T_{2}\in \FF $. Since $ G\not \in \FF $, we get that $ |T_{2}|>2  $, as claimed.  Let $ A $ be a normal  subgroup of $ P $ of order $ 2 $ contained in $ T_{2} $. Then $ A<T_{2} $ and  $  AT_{1} $ has  order $ d $. By hypothesis,  $ AT_{1} $ satisfies the partial $ \Pi $-property in $ G $. By Lemma \ref{pass},  $ G $ has  a chief series $$ \varOmega_{G}: 1 =G^{*}_{0} < G^{*}_{1} < G^{*}_{2}=T_{1}T_{2} < \cdot\cdot\cdot < G^{*}_{n}= G $$  passing through $ T_{1}T_{2} $ such that $ |G:N_{G}(AT_{1}G^{*}_{i-1}\cap G^{*}_{i})| $ is a $ p $-number  for every $ G $-chief factor $ G^{*}_{i}/G^{*}_{i-1} $ $ (1\leq i\leq n) $ of $ \varOmega_{G} $. Note that every chief factor of $ G $ below $ T_{1}T_{2} $ has order either $ |T_{1}| $ or $ |T_{2}| $. Then $ |G:N_{G}(AT_{1}\cap G^{*}_{1})| $ is a $ p $-number. Thus $ AT_{1}\cap G^{*}_{1}\unlhd G $. By similar reasoning, we have that $ AT_{1}G^{*}_{1}=AT_{1}G^{*}_{1}\cap T_{1}T_{2}\unlhd G $ and thus  $ AT_{1}G^{*}_{1}=T_{1}T_{2} $. If $ AT_{1}\cap G^{*}_{1}=1 $, then $ 2|T_{1}|=2|G^{*}_{1}|=|T_{2}| $. This is impossible. If $ AT_{1}\cap G^{*}_{1}=G^{*}_{1} $, then $ AT_{1}\unlhd G $, and thus $ AT_{1}=T_{1}T_{2} $. It means that $ |T_{2}|=2 $, a contradiction.
		
		Therefore $ s = 1 $. By Step 1,  $ F(G)=O_{p}(G) $ is a minimal normal subgroup of $ G $.
		
		\vskip0.1in
		
		\noindent\textbf{Step 5.} $ |P : O_{p}(G)| = p $.
		
		\vskip0.1in
		
		Let $ L $ be a maximal normal subgroup of $ G $ containing $ O_{p}(G) $. Since $ O^{p'}(G) = G $, it follows that  $ |G : L| = p $. By Step 2, $ d <|P|/p=|P\cap L| $. By Lemma \ref{subgroup}, the hypotheses are inherited by $ L $. Then $ L\in \FF $ by the choice  of $ G $. Since $ O_{p'}(G)=1 $ by Step 1, we have that $ P\cap L\unlhd G $. In view of  Step 4, $ O_{p}(G)=P\cap L $. Therefore,  $ |P : O_{p}(G)| = p $, as wanted.

		\vskip0.1in
		
		\noindent\textbf{Step 6.}  The final contradiction.
		
		\vskip0.1in
		
		By Step 4 and Lemma \ref{order-d}(1), we have $ |O_{p}(G)| \leq d < |P| $. By Step 5, $ |O_{p}(G)| = d $ and $ |P| = dp $, which  contradicts Step 2. This final contradiction completes the proof.
	\end{proof}

	\begin{proof}[Proof of Theorem \ref{third}]
		Assume  that the theorem is not ture and $ G $ is a counterexample of  minimal  order. Let $ \FF $ denote the class of all $ p $-soluble groups whose $ p $-length is at most $ 1 $. We proceed via the following steps.

		\vskip0.1in
		
		\noindent\textbf{Step 1.}  $ O_{p'}(G) = 1 $ and $ O^{p'}(G) = G $. In particular, $ F(G)\leq O_{p}(G) $.
		
		\vskip0.1in
		
		By Lemma \ref{over} and Lemma \ref{subgroup}, $ G/O_{p'}(G) $ and $ O^{p'}(G) $ both satisfy the hypotheses of the theorem. If $ O_{p'}(G)> 1 $, then  $ G/O_{p'}(G) $ is $ p $-soluble by the choice of $ G $. This forces that  $ G $ is $ p $-soluble, a contradiction. Hence $ O_{p'}(G)=1 $ and $ F(G)\leq O_{p}(G) $.  If $ O^{p'}(G)<G $, then $  O^{p'}(G) $ is $ p $-soluble, and thus $ G $ is $ p $-soluble. This contradiction shows that $ O^{p'}(G)=G $.
		
		\vskip0.1in
		
		\noindent\textbf{Step 2.}  Let $ L $ be a proper subgroup of $ G $. Then either $ |L|_{p} \leq d $ or $ L $ is $ p $-soluble. Furthermore, if $ L $ is a proper normal subgroup of $ G $, then either $ |P\cap L| \leq d $ or $ P\cap L  \unlhd G $.
		
		\vskip0.1in
		
		Assume that $ L $ is a proper subgroup of $ G $ and $ |L|_{p} > d $. Then $ L $ satisfies the hypotheses of the theorem by Lemma \ref{subgroup}. The minimal choice of $ G $ implies that $ L $ is $ p $-soluble.   If $ L $ is a proper normal subgroup of $ G $ and $ |P\cap L| > d $, then $ L $ is a $ p $-soluble group.  According to  Theorem \ref{second}, $ G $  has $ p $-length  at most $ 1 $. Since $ O_{p'}(G)=1 $, we conclude that $ P\cap L\unlhd G $.
		
		\vskip0.1in

		\noindent\textbf{Step 3.} $ p^{2}<d<|P|/p^{2} $.
		
		\vskip0.1in
		
		Applying Theorem \ref{maximal} and Theorem \ref{minimal}, we have  $ d\not =p,  |P|/p $. By \cite[Theorem 1.3 and Theorem 1.4]{Qiu-Liu-Chen},   $ d\not =p^{2},  |P|/p^{2} $.
		
		\vskip0.1in

		\noindent\textbf{Step 4.} Every minimal normal subgroup of $ G $ is a $ p $-group of order $ d $.
		
		\vskip0.1in
		
		Let $ T $ be a minimal normal subgroup of $ G $. By Lemma \ref{order-d} and Step 1, $ T $ is a $ p $-group of order at most $ d $. Assume that $ |T| < d $ and we work to obtain a contradiction.  If either $ p $ is odd or $ p = 2 $ and $ |T| < d/2 $, then $ G/T $ satisfies the hypotheses of the theorem. Hence, $ G/T $ is $ p $-soluble and so is $ G $, which is  absurd.  Therefore,  we may assume that $ p = 2 $ and $ |T|=d/2>2 $ for every minimal normal subgroup $ T $ of $ G $. This means that every subgroup of $ P/T $ of order $ 2 $ satisfies the
		partial $ \Pi $-property in $ G $.  Furthermore, $ P/T $ is not quaternion-free by Theorem \ref{minimal}.

		Suppose that $ R $ is a  minimal normal subgroup of $ G $ which is different from $ T $. Then $ |R|=d/2>2 $. Let $ A  $ be a normal subgroup  of $ P $ of order $ 2 $ such that $ A\leq T $. Then $ A R $ has order $ d $. By hypothesis, $ A R $ satisfies the partial $ \Pi $-property in $ G $. By Lemma \ref{pass}, $ G $ has  a chief series $$ \varOmega_{G}: 1 =G^{*}_{0} < G^{*}_{1} <G^{*}_{2}=TR < \cdot\cdot\cdot < G^{*}_{n}= G $$  passing through $ TR $ such that $ |G:N_{G}(A RG_{i-1}\cap G^{*}_{i})| $ is a  $ 2 $-number  for every $ G $-chief factor $ G^{*}_{i}/G^{*}_{i-1} $ $ (1\leq i\leq n) $ of $ \varOmega_{G} $. In particular, $ |G:N_{G}(AR\cap G_{1}^{*})| $ is a $ 2 $-number.  Hence $ A R\cap G^{*}_{1}\unlhd G $. Notice that $ |G^{*}_{1}|=d/2=|TR/G^{*}_{1}| $. If $ A R\cap G^{*}_{1}=1 $, then $ d^{2}/2=|ARG^{*}_{1}|=|RG^{*}_{1}|=d^{2}/4 $, a contradiction. Therefore, $ A R\cap G^{*}_{1}=G^{*}_{1} $. For  the $ G $-chief factor $ G^{*}_{2}/G^{*}_{1} $, $ |G:N_{G}(ARG^{*}_{1}\cap G^{*}_{2})|=|G:N_{G}(AR)| $ is a $ 2 $-number, and thus  $ AR\unlhd G $. This implies that $ AR=TR $. Therefore $ d=4 $, a contradiction. It means that $ T $ is the unique minimal normal subgroup of $ G $.

		Now pick a cyclic subgroup $ X/T $  of $ P/T $ of order $ 4 $. Then $ |X|=2d $. If $ T\leq \Phi(X)  $, then $ X $ is cyclic. This yields  $ |T| = 2 $ and  $ d = 4 $, a contradiction. Hence we may suppose that $ T\nleq \Phi(X) $. There exists a maximal subgroup $ V $ of $ X $ such that $ X = VT $. Then $ V $ has order $ d $ and satisfies the partial $ \Pi $-property in $ G $.  Note that $ T $ is the unique minimal normal subgroup of $ G $. By a routine check, we see that $ X/T = VT/T $ satisfies the partial $ \Pi $-property in $ G/T $.  Now applying Theorem \ref{minimal},  $ G $ is $ 2 $-supersoluble,  a contradiction.  Hence  Step 4 holds.

		\vskip0.1in
		
		\noindent\textbf{Step 5.} Let $ K $ be a maximal normal subgroup of $ G $. Let $ S $ denote a Sylow $ p $-subgroup of $ K $. Then $ S $ is normal in $ G $. In particular, $ K $ is $ p $-closed,  and so $ K $ is $ p $-soluble.
		
		\vskip0.1in
		
		Let $ T $ be a minimal normal subgroup of $ G $ contained in $ K $. By Step 4, $ |T| = d $ and
		then $ |S| \geq d $. If $ |S| = d $, then $ T=S  $, and so $ K $ is $ p $-closed. If $ |S| > d $, then  $ S\unlhd G $ by Step 2. Therefore $ K $ is a $ p $-closed.
		
		\vskip0.1in
		
		\noindent\textbf{Step 6.} $ G $ has a unique maximal normal subgroup, $ K $ say. Moreover,  $ G/K $ is a non-abelian simple group of order divisible by $ p $.
		
		\vskip0.1in
		
		Suppose that $ B $ is a   maximal normal subgroup of $ G $ which is different from $ K $. Then $ K $  and $ B $ are $ p $-soluble by Step 5. Hence $ G = KB $ is $ p $-soluble, a contradiction. Hence $ K $  is the  unique maximal normal subgroup of $ G $. Since $ K $ is $ p $-soluble by Step 5, we have $ G/K $ is a non-abelian simple group. In view of  Step 1, $ p $ divides $ |G/K| $.

		\vskip0.1in
		
		\noindent\textbf{Step 7.}  Write $ d = p^{n} $. Every $ p $-subgroup $ Y $ of $ P $ of order $ |Y| = p^{m} $ $ ( m \leq n ) $ is  contained in $ K $.
		
		\vskip0.1in
		
		Let $ Y $ be a $ p $-subgroup of $ P $ such that $ |Y| = p^{m} $  $ (m\leq n) $. Assume that $ Y\nleq K $. Let $ H $ be a subgroup of $ P $ of  order $ d=p^{n} $ containing $ Y $. By hypothesis, $ Y $ satisfies the partial $ \Pi $-property in $ G $. In view of Step 6,  the chief factor $ G/K $
		occurs in every chief series of $ G $. Then $ |G:N_{G}(YK\cap G)|=|G:N_{G}(YK)| $ is a $ p $-number. Thus $ G=PN_{G}(YK) $ and $ (YK)^{G}=Y^{P}K $.  If $ Y^{P}K = G $, then $ G/K $ is a $ p $-group, and so $ G $ is $ p $-soluble by Step 5. This is a contradiction. If $ Y^{P}K = K $, then $ Y\leq K $, a contradiction. Thus Step 7 follows.

		\vskip0.1in
		
		\noindent\textbf{Step 8.}  $ K= \Phi(G) = O_{p}(G)=F(G) $ and $ G=O^{p}(G)=G' $.

		\vskip0.1in
		
		By Step 5, $ G = O^{p}(G) = G' $. Since $ O_{p'}(G) = 1 $, it follows that $ \Phi(G)\leq O_{p}(G) $. Suppose that $ K $ is not contained in $ \Phi(G) $. There exists a proper subgroup $ M $ of $ G $ such that $ G = KM $. Since $ p $  divides $ |G/K| $ by  Step 6, we deduce  that  $ K $ does not contain any Sylow $ p $-subgroup of $ M $. By Step 7, the order of the Sylow $ p $-subgroups of $ M $ is greater than $ d $. Thus $ M $ is $ p $-soluble by Step 2. Then $ G $ is $ p $-soluble, a contradiction. Hence $ K\leq  \Phi(G) $. Furthermore, $ K = \Phi(G) = O_{p}(G) = F(G) $.

		\vskip0.1in
		
		\noindent\textbf{Step 9.} The final contradiction.
		
		\vskip0.1in
		
		By Step 6 and Step 8, $ O_{p}(G) $ is the unique maximal normal subgroup of $ G $. By Step 4, $ {\rm Soc}(G)\leq O_{p}(G) $. In view of Lemma \ref{p-supersoluble}, we deduce  that $ G $ is $ p $-soluble, a contradiction. This final contradiction completes the proof.
	\end{proof}

	\begin{proof}[Proof of Theorem \ref{first}]
		Combining Theorem  \ref{second} and Theorem \ref{third}, we conclude that $ G $ is  $ p $-soluble with   $ p $-length  at most $ 1 $.
	\end{proof}

	\section*{Acknowledgments}
	
	\hspace{0.5cm} This work is supported by the National Natural Science Foundation of China (Grant No.12071376, 11971391) and  the  Fundamental Research Funds for the Central Universities (No. XDJK2020B052).

    \small


\end{document}